\documentclass[12pt,a4paper]{amsart}
\usepackage{amsmath,amssymb,amsthm,enumerate,a4wide,}
\usepackage[pagebackref,colorlinks,linkcolor=blue,citecolor=red,
urlcolor=blue,hypertexnames=true]{hyperref}
\usepackage{amsrefs}

\title{The semisimple conjugacy classes\\ in the symplectic groups}
\author{G.E. Wall}
\address{Emeritus Professor G.E. Wall, School of Mathematics and Statistics,
The University of Sydney, NSW 2006, Australia.}
\date{\today}

\theoremstyle{plain}
\newtheorem{thm}{Theorem}[section]
\newtheorem{cor}[thm]{Corollary}
\newtheorem{lemma}[thm]{Lemma}

\theoremstyle{definition}
\newtheorem{defn}[thm]{Definition}
\newtheorem{notation}[thm]{Notation}
\newtheorem{assump}[thm]{Assumption}

\newtheorem{obs}[thm]{Observation}

\DeclareMathOperator{\Sp}{Sp}
\newcommand{\spl}[2]{\Sp_{#1}(#2)}
\DeclareMathOperator{\GL}{GL}
\newcommand{\gl}[2]{\GL_{#1}(#2)}
\DeclareMathOperator{\charc}{char}

\newcommand{\K}{\mathcal{K}}

\makeatletter
\def\paragraph{\@startsection{paragraph}{4}%
  \z@{.3\linespacing\@plus.5\linespacing}{-\fontdimen2\font}%
  {\normalfont\bfseries}}
\makeatother

\begin{document}

\maketitle

\begin{abstract}
  We determine the conjugacy classes of semisimple elements in the
  symplectic groups $\spl{2m}{F}$, where $F$ is an arbitrary field of
  characteristic not $2$. This note was originally a letter dated 23 March, 2006, from
  G.E. Wall to Cheryl Praeger, and has been reproduced with his
  kind permission.
\end{abstract}

\section{The general problem}

The problem in question is to determine the conjugacy classes in the
symplectic groups $\spl{2m}{F}$ over a field $F$. The general method
proposed in the present section is used in Section $2$ to give a
detailed (and elementary) account of the conjugacy classes of semisimple
elements in the case where $\charc F \neq 2$.  (A \emph{semisimple}
element is one whose minimal polynomial is separable.  These include all
elements of finite order when $\charc F = 0$.)

Denote by $\mathcal{F}$ the set of all non-degenerate alternating
bilinear forms
\begin{displaymath}
  f\colon F^{2m} \times F^{2m} \rightarrow F
\end{displaymath}
and by $\mathcal{G}$ the general linear group $\gl{2m}{F}$ of all
nonsingular linear mappings
\begin{displaymath}
  T\colon F^{2m} \rightarrow F^{2m}.
\end{displaymath}
The natural permutation action of $\mathcal{G}$ on $\mathcal{F}$ is
defined by
\begin{displaymath}
  (fT)(u, v) = f(uT, vT) \text{ for all } u, v \in F^{2m}.
\end{displaymath}

The subgroup of $\mathcal{G}$ formed by those elements that fix a given
$f$ is the symplectic group $\spl{}{f}$. Since $\mathcal{G}$ acts
transitively on $\mathcal{F}$, these symplectic groups form a complete
set of conjugate subgroups of $\mathcal{G}$ (thereby justifying the
generic notation $\spl{2m}{F}$).

In order to put forms and linear mappings on the same footing, we
introduce the set of pairs
\begin{displaymath}
  \mathcal{P} = \{(f, T) \mid f \in \mathcal{F}, T \in
  \spl{}{f}\}
\end{displaymath} and define the action of $\mathcal{G}$ on $\mathcal{P}$ by
\begin{displaymath}
  (f, T)S = (fS, S^{-1}TS).
\end{displaymath}

The crucial observation is this:
\begin{obs}
  For fixed $f_0 \in \mathcal{F}$, the elements $T_1, T_2, \ldots \in
  \mathcal{G}$ are a set of representatives for the conjugacy classes of
  $\spl{}{f_0}$ if and only if $(f_0, T_1), (f_0, T_2), \ldots \in
  \mathcal{P}$ are a set of representatives for the orbits under the
  action of $\mathcal{G}$ on $\mathcal{P}$.
\end{obs}

In short, the original conjugacy class problem can be reformulated as
one about orbits on $\mathcal{P}$. It is from this new viewpoint that
the problem will be treated from here. We now describe an alternative
way of constructing a set of orbit representatives.

\paragraph{Step 1:} We first choose representative elements $R_1, R_2,
\ldots$ from the conjugacy classes $\K_1, \K_2, \ldots$ of
$\mathcal{G}$, determining at the same time their centralisers $b_1,
b_2, \ldots$ in $\mathcal{G}$. This is a matter of standard linear
algebra.

\paragraph{Step 2:} We next determine, for each such representative
$R_k$, the set
\begin{displaymath}
  \mathcal{F}_k = \{f \in \mathcal{F} \mid fR_k = f\} = \{f \in
  \mathcal{F} \mid (f, R_k) \in \mathcal{P}\}.
\end{displaymath}
It may happen that $\mathcal{F}_k$ is empty, which simply means that no
symplectic group contains elements of $\mathcal{G}$ conjugate to
$R_k$. Assume now that $\mathcal{F}_k$ is nonempty.
\paragraph{Step 3:}
The centraliser $b_k$ acts naturally as a permutation group on
$\mathcal{F}_k$. The final step is to determine a set of representatives
$f_{k 1}, f_{k 2}, \ldots$ for the orbits of $b_k$ in this action.

The pairs
\begin{displaymath}
  (f_{1 1}, R_1), (f_{1 2}, R_1), \ldots, (f_{2 1}, R_2), (f_{2 2}, R_2), \ldots
\end{displaymath} 
so constructed form an alternative set of representatives for the orbits
under the action of $\mathcal{G}$ on $\mathcal{P}$, and thus give a new
way of determining the conjugacy classes in the symplectic group.

\section{Semisimple elements}

In order to avoid exceptional cases, we assume throughout that
\begin{equation}
  \charc F \neq 2.
\end{equation}
No further restriction is imposed for the present.

\medskip

The first task (Step 2 of \S 1) is as follows: given a nonsingular,
even-dimensional linear transformation over the field $F$, it is
required to determine the nonsingular alternating bilinear forms that it
leaves invariant.

In matrix terms, we are given $X \in \gl{2m}{F}$ and are required to
determine those $A \in \gl{2m}{F}$ such that
\begin{equation}\label{2}
  A = -A', \quad A = XAX',
\end{equation}
where $'$ denotes transpose. Notice that these conditions are equivalent
to 

(i) the form $f_A$ given by $f_A(u, v) = uAv'$ lies in $\mathcal{F}$,
and 

(ii) $X$ leaves $f_A$ invariant, so that $(f_A, X) \in \mathcal{P}$.

\medskip

The second task (Step 3 of \S 1) arises when the set of $A$ in \eqref{2}
is nonempty. The centraliser of $X$ in $\gl{2m}{F}$ acts on this set by
congruence:
\begin{equation}\label{3}
  A \mapsto YAY' \text{ for } Y \text{ such that } Y^{-1}XY = X,
\end{equation}
and it is required to determine a set of representatives
\begin{equation}\label{4}
  A_1, A_2, \ldots
\end{equation}
for the orbits. It is tacitly assumed from now on that $A$ and $X$ are
nonsingular matrices satisfying \eqref{2}.

\begin{lemma}\label{lem:1}
  $X$ is similar to $X^{-1}$.
\end{lemma}

\begin{proof}
  By \eqref{2}, $A^{-1}X^{-1}A = X'$, so that $X^{-1}$ is similar to
  $X'$ and hence to $X$.
\end{proof}

\begin{notation}
  If $f(t)$ is a monic polynomial with $f(0) \neq 0$ then $f^-(t)$
  denotes the monic polynomial whose roots are the reciprocals of those
  of $f(t)$.  Let $c_Y(t)$ denote the characteristic polynomial of a
  square matrix $Y$.
\end{notation}

\begin{defn}
  An \emph{elementary divisor} of a square matrix $Y$ is a divisor of
  the minimal polynomial of $Y$ of the form $f(t) = g(t)^\lambda$, where
  $g(t)$ is monic and irreducible, which is related to the rational
  canonical form of $Y$. 
\end{defn}

Later we shall assume that $Y$ is semisimple.
In this case, irreducible factors of the minimal polynomial of $Y$ occur
with multiplicity 1.


\begin{cor}
  If $f(t)$ is an elementary divisor of $X$, then $f^-(t)$ is an
  elementary divisor of the same multiplicity.
\end{cor}

Suppose that $X$ has block diagonal form, and $A$ has corresponding
block matrix form:
\begin{equation}\label{5}
  X = \left( \begin{array}{ccc}
      X_1 & 0 & \cdots \\
      0 & X_2 & \\
      \vdots & & \ddots \end{array} \right), 
  \quad
  A = \left( \begin{array}{ccc}
      A_{11} & A_{12} & \cdots \\
      A_{21} & A_{22} & \\
      \vdots & & \ddots \end{array} \right).
\end{equation}

Then, by \eqref{2},
\begin{equation}\label{6}
  X_i A_{ij} X_j' = A_{ij} = -A'_{ji}  \text{ for all } i, j.
\end{equation}

Hence $A_{ij} X_j' = X_i^{-1}A_{ij}$ and so, more generally,
\begin{equation}\label{7}
  A_{ij}f(X_j)' = f(X_i^{-1})A_{ij}
\end{equation}
for any polynomial $f(t)$.
\begin{notation}
Let $c_Y(t)$ denote the characteristic polynomial of the square matrix $Y$.
\end{notation}

\begin{lemma}\label{lem:2}
  If
  \begin{equation}\label{8}
    (c_{X_i^{-1}}(t), c_{X_j}(t)) = 1,
  \end{equation}
  then $A_{ij} = 0$.
\end{lemma}

\begin{proof}
  Taking $f(t) = c_{X_i^{-1}}(t)$ in \eqref{7}, we get $A_{ij}
  c_{X_i^{-1}}(X_j)' = 0$. However, in view of \eqref{8},
  $c_{X_i^{-1}}(X_j)$ is nonsingular, whence $A_{ij} = 0$.
\end{proof}

Elementary divisors $f_1(t)$, $f_2(t)$ of $X$ are powers of irreducible
monic polynomials $g_1(t)$, $g_2(t)$. We say that $f_1(t)$ and $f_2(t)$
are \emph{related} if $g_2(t) = g_1(t)$ or $g_1^-(t)$.

By the theory of elementary divisors, we may choose the blocks $X_i$ in
\eqref{5} in such a way that elementary divisors $f_1(t)$, $f_2(t)$ of
$X$ are elementary divisors of the same $X_i$ if, and only if, they are
related. With such a choice of the $X_i$, Lemma~\ref{lem:2} shows that
$A$ has corresponding block diagonal form
\begin{displaymath}
  \left( \begin{array}{ccc}
      A_{11} & 0 & \cdots \\
      0 & A_{22} & \\
      \vdots & & \ddots \end{array} \right).
\end{displaymath}
In this way the original problem for $X$ is reduced to the same problem
for the individual blocks $X_i$. We may therefore assume:

\begin{assump}\label{assump}
  There exists a monic irreducible polynomial $g(t) \neq t$ such that
  every elementary divisor of $X$ is a power of $g(t)$ or $g^-(t)$.
\end{assump}

\paragraph{Case 1:} $\mathbf{g(t) \neq g^-(t)}$.
We may assume in \eqref{5} that
\begin{displaymath}
  X = \left( \begin{array}{cc} X_1 & 0 \\ 0 & X_2 \end{array}
  \right),
\end{displaymath}
where $c_{X_1}(t)$, $c_{X_2}(t)$ are powers of $g(t)$, $g^-(t)$
respectively. By Lemma~\ref{lem:1}, $X_2$ is similar to $X_1^{-1}$. We
may therefore assume further that
\begin{equation}\label{9}
  X = \left( \begin{array}{cc}
      X_1 & 0 \\
      0 & (X_1^{-1})' \end{array} \right), \quad
  A = \left( \begin{array}{cc}
      A_{11} & A_{12} \\
      -A_{12}' & A_{22} \end{array} \right).
\end{equation}

By Lemma~\ref{lem:2}, $A_{11} = A_{22} = 0$. Also, by \eqref{6},
$X_1A_{12}X_1^{-1} = A_{12}$, i.e. $X_1$ commutes with $A_{12}$.

Now let
\begin{equation}\label{10}
  Y = \left( \begin{array}{cc}
      A_{12} & 0 \\
      0 & I_m \end{array} \right), \quad
  J = \left( \begin{array}{cc}
      0 & I_m \\
      -I_m & 0 \end{array} \right),
\end{equation}
where $I_m$ is the $m \times m$ unit matrix, and $m = \mathrm{deg} \,
c_{X_1}(t)$. Then
\begin{displaymath}
  Y^{-1}XY = X, \quad YJY' = A,
\end{displaymath}
showing that, with $X$ as in \eqref{9}, there is just one orbit under
the action \eqref{3}, represented by the matrix $J$ in
\eqref{10}. Expressed differently, the conjugacy class of $X$ in
$\gl{2m}{F}$ intersects each symplectic subgroup in a single conjugacy
class of the latter.

\bigskip

\paragraph{Case 2:} $\mathbf{g(t) = g^-(t)}$.

In general, the elementary divisors of $X$ may be arbitrary powers of
$g(t)$ with arbitrary multiplicities. We now impose the condition that
$X$ be semisimple:

\begin{assump}
  $X\in\gl{2m}{F}$ has the single irreducible elementary divisor
  $g(t)\ne t$ with multiplicity $n$.
\end{assump}

We may therefore assume that
\begin{equation}\label{11}
  X = \mathrm{diag}(\underset{n}{\underbrace{R, \ldots, R}}),
\end{equation}
where
\begin{equation}\label{12}
  c_R(t) = g(t).
\end{equation}

Since $c_R(t)$ is irreducible, the matrices
\begin{displaymath}
  f(R) \quad (f(t) \in F[t])
\end{displaymath}
form a field
\begin{displaymath}
  K \cong F[t]/g(t)F[t]
\end{displaymath}
and every matrix that commutes with $R$ is in $K$. It follows that the
centralizer of $X$ in $\gl{2m}{F}$ consists of the nonsingular $n \times
n$ block matrices
\begin{equation}\label{13}
  B = (f_{ij}(R))_{i,j = 1, \ldots, n},
\end{equation}
for polynomials $f_{ij}(t) \in F(t)$.  These matrices form a group that
we may identify with $\gl{n}{K}$.

If $\mathrm{deg} \, g(t) = 1$, then $g(t) = t \pm 1$ (since $g(t) =
g^-(t)$), the matrix $R$ is a $1 \times 1$ matrix $(\pm 1)$ and $X = \pm
I_{2m}$. The centralizer of $X$ is $\gl{2m}{F}$ and the nonsingular $2m
\times 2m$ skew-symmetric matrices form a single orbit under its
action. We assume from now on that $\mathrm{deg} \, g(t) \geq 2$.

Let
\begin{equation}\label{14}
  A = (A_{ij})_{i,j = 1, \ldots, n}
\end{equation}
be the block form of $A$ corresponding to \eqref{11}. The equation $A =
XAX'$ in \eqref{2} is then equivalent to the set of equations
\begin{equation}\label{15}
  RA_{ij}R' = A_{ij}.
\end{equation}
Now, since $g(t) = g^-(t)$, $R$ is similar to $R^{-1}$ and so
\begin{equation}\label{16}
  R' = T^{-1}R^{-1}T
\end{equation}
for some $T \in \gl{2m/n}{F}$. Thus, we may rewrite \eqref{15} as
\begin{displaymath}
  R(A_{ij}T^{-1}) = (A_{ij}T^{-1})R,
\end{displaymath}
whence $A$ has the form
\begin{equation}\label{17}
  A = (f_{ij}(R)T).
\end{equation}
We write this equation as
\begin{equation}\label{18}
  A = B{\mathcal T}\quad
  \text{where $B=(f_{ij}(R))$ and }{\mathcal T}
  =\mathrm{diag}(\underset{n}{\underbrace{T, \ldots, T}}).
\end{equation}

Now, the mapping $K \rightarrow K$ defined by
\begin{displaymath}
  \phi(R) \mapsto \phi(R^{-1}) \quad (\phi(t) \in F[t])
\end{displaymath}
is a field automorphism of $K$ of order $2$, since $R\ne R^{-1}$. For a
matrix
\begin{displaymath}
  Y = (\phi_{ij}(R)) \in \gl{n}{K},
\end{displaymath}
we define
\begin{displaymath}
  Y^{\ast} = (\phi_{ij}(R^{-1}))^{\mathrm{tr}},
\end{displaymath}
where $\mathrm{tr}$ denotes transpose qua $n \times n$ matrix over $K$
and \emph{not} qua $2m \times 2m$ matrix over~$F$,
i.e. $Y^*=(\phi_{ji}(R^{-1}))$.  Accordingly, $Y$ is called
\emph{Hermitian} if $Y^\ast = Y$ and $Y_1, Y_2$ are said to be
\emph{$\ast$-congruent} if $Y_2 = C Y_1 C^\ast$ for some $C \in
\gl{n}{K}$.

The following two results are proved by routine calculations:
\begin{lemma}
  \begin{enumerate}
  \item[(i)] If $A=B{\mathcal T}$ as in~\eqref{18} and $Y \in
    \gl{n}{K}$, then $YAY'= YBY^\ast{\mathcal T}$.
  \item[(ii)] If $A=B{\mathcal T}$ as in~\eqref{18} and $T$ is
    skew-symmetric, then $A$ is skew-symmetric if, and only if, $B$ is
    Hermitian.
  \end{enumerate}
\end{lemma}

\begin{proof}
  \noindent {\emph{(i)}} We may write $Y$ as a block matrix
  $(\phi_{ij}(R))$ for some $\phi_{ij}(t) \in F[t]$. Then
  \begin{displaymath}
    \begin{array}{rl}
      YAY' & =  (\phi_{ij}(R))(f_{ij}(R)T)(\phi_{ij}(R))' \\& = 
      (\sum_{\lambda, \mu} \phi_{i \lambda}(R) f_{\lambda \mu}(R)T \phi_{j
        \mu}(R')) \\ & = 
      (\sum_{\lambda,\mu} \phi_{i \lambda}(R) f_{\lambda \mu}(R) \phi_{j
        \mu}(R^{-1})T), \\
      & = YBY^*{\mathcal T}\end{array}
  \end{displaymath} by \eqref{16}.

  \noindent {\emph{(ii)}} By \eqref{18}, $A'={\mathcal T}'B'$, and since
  $T$ is skew-symmetric we have
  \begin{displaymath}
    \begin{array}{rll}
      -A' & = (-T'f_{ji}(R'))
      = (Tf_{ji}(R)') \quad \quad & \text{ (as $T' = -T$)}\\
      & = (f_{ji}(R^{-1})T) \quad \quad & \text{ (by \eqref{16}).}
    \end{array}
  \end{displaymath}
  
  So $A = -A'$ if, and only if, $f_{ij}(R)T = f_{ji}(R^{-1})T$ for all
  $i, j$. Since $T$ is invertible, this holds if, and only if,
  \begin{displaymath}
    B = (f_{ij}(R))  = (f_{ji}(R^{-1}))= B^\ast.
  \end{displaymath}
\end{proof}

Thus, \emph{provided that $T$ can be chosen skew-symmetric}, our
conjugacy class problem in the present case reduces to a classification
problem for Hermitian forms over the extension $K$ of $F$. Since $T$ can
obviously be replaced in \eqref{17} by $h(R)T$, where $h(R)$ is any
nonzero (and hence nonsingular) element of $K$, the following result
shows that such a choice of $T$ is always possible.

\begin{lemma}\label{lem:3}
  If $g(t) \neq t \pm 1$, then there exists a nonzero element $h(R)$ of
  $K$ such that $h(R)T$ is skew-symmetric.
\end{lemma}

\begin{proof} 
  \eqref{16} can be written $RTR' = T$. Transposing, we get $RT'R' =
  T'$, whence
  \begin{equation}\label{19}
    R' = (T')^{-1}R^{-1}(T').
  \end{equation}
  Comparing with \eqref{16}, we deduce that $T'T^{-1}$ commutes with
  $R$, whence
  \begin{displaymath}
    T' = f(R)T
  \end{displaymath}
  for some $f(R) \in K$.

  Now, if both $RT$ are $T$ were symmetric, we would have
  \begin{displaymath}
    RT = (RT)' = T'R' = TR' = R^{-1}T
  \end{displaymath}
  by \eqref{16}.  But this implies that $R = R^{-1}$ and so $R^2 = I$,
  contrary to the assumption that $g(t) \neq t \pm 1$.

  It follows that at least one of $T$ and $RT$ --- let us say $T$ itself
  --- is \emph{not} symmetric. But then
  \begin{displaymath}
    T - T' = (1-f(R))T
  \end{displaymath}
  is nonzero and skew-symmetric, as required.
\end{proof}

\section{Summary}

We wish to determine a complete, irredundant set of conjugacy class
representatives for the semisimple elements of the symplectic groups
$\spl{2m}{F}$, where $F$ is an arbitrary field of characteristic not
$2$. To do so, it suffices to consider the elements whose characteristic
polynomial is divisible only by powers of $g(t)$ and $g^-(t)$, for some
irreducible $g(t) \in F[t]$.  We first choose a set $R_1, R_2, \ldots,
R_k$ of representatives of conjugacy classes of such elements in
$\gl{2m}{F}$, and discard any such $R_i$ that do not preserve a
symplectic form.

In Case 1 of Section~2, $g(t) \neq g^-(t)$. Then the
$\gl{2m}{F}$-conjugacy class of $R_i$ meets $\spl{2m}{F}$ in a unique
conjugacy class.

In Case 2, $g(t) = g^-(t)$. The only $R_i$ for which $g(t)$ has degree
$1$ are $\pm I_{2m}$.  Let $R_i = X$ be as in \eqref{11}, where $R$ has
characteristic polynomial $g(t)$ of degree greater than $1$, and let $K$
denote the field isomorphic to the set of all polynomials in $R$.  Then
each congruence class of Hermitian forms on $K^{2m/n}$ corresponds to an
$\spl{2m}{F}$-conjugacy class of matrices that are similar to $R_i$. In
particular, if $F$ is finite then there is only one such class, and so
once again the $\gl{2m}{F}$-conjugacy class of $R_i$ meets $\spl{2m}{F}$
in a unique conjugacy class.

\end{document}